\documentclass[graybox]{svmult}

\usepackage{type1cm}        
\usepackage{makeidx}         
\usepackage{graphicx}        
\usepackage{multicol}        
\usepackage[bottom]{footmisc}

\usepackage{newtxtext}       %
\usepackage[varvw]{newtxmath}       

\makeindex             

\begin{document}

\title*{On asymptotics of ring $Q$-homeomorphisms with respect to $p$-modulus near the origin}
\author{Ruslan Salimov and Bogdan Klishchuk}
\institute{Ruslan Salimov \at Ruslan Salimov, Institute of Mathematics of NAS of Ukraine, \email{ruslan.salimov1@gmail.com}
\and Bogdan Klishchuk \at Bogdan Klishchuk, Institute of Mathematics of NAS of Ukraine \email{kban1988@gmail.com}}
%
%
\maketitle

\abstract*{We consider the class of ring $Q$-homeomorphisms with respect to
$p$-modulus in $\mathbb{R}^{n}$ with $p > n$, and obtain lower bounds for
limsups of the distance distortions under such mappings. These estimates can be treated
as H\"{o}lder's continuity of the inverses near the origin.
The sharpness is illustrated by example.}

\abstract{We consider the class of ring $Q$-homeomorphisms with respect to
$p$-modulus in $\mathbb{R}^{n}$ with $p > n$, and obtain lower bounds for
limsups of the distance distortions under such mappings. These estimates can be treated
as H\"{o}lder's continuity of the inverses near the origin.
The sharpness is illustrated by example.}

\section{Introduction}
\label{sec:1}

In this paper, we study the distortion of distances under ring $Q$-homeomorphisms
with respect to $p$-modulus in $\mathbb{R}^{n}$ with $p>n$. For such
mappings we derive lower bounds for upper limits.

Let us recall some definitions, see \cite{Vai}. Let $\Gamma$ be a family of curves
$\gamma$ in $\mathbb{R}^{n}$, $n\geqslant2$. A Borel measurable function
$\rho:{\mathbb{R}^{n}}\to[0,\infty]$ is called {\it admissible} for
$\Gamma$, (abbr. $\rho\in{\rm adm}\,\Gamma$), if
$$
 \int\limits_{\gamma}\rho(x)\,ds\
\geqslant\ 1
$$
for all locally rectifiable $ \gamma\in\Gamma$. Let $p\in (1,\infty)$.

The quantity
$$
M_p(\Gamma)\ =\ \inf_{\rho\in\mathrm{adm}\,\Gamma}\int\limits_{{\mathbb{R}^{n}}}\rho^p(x)\,dm(x)\,
$$
is called {\it $p$--modulus} of the family $\Gamma$. Here $dm(x)$ corresponds to the Lebesgue
measure in $\mathbb{R}^{n}$.

Let $Q:D\to[0,\infty]$ be a measurable function. Recall that a homeomorphism
$f:D\to \mathbb{R}^n$ is said to be a {\it $Q$-homeomorphism with respect to
$p$-modulus} if \begin{equation*} \label{eq13.1}
 \mathrm{M}_{p}(f\Gamma)\leqslant\int\limits_{D}
Q(x)\,\rho^{p}(x)\ dm(x)\end{equation*}  for every family $\Gamma$ of
paths in $D$ and every admissible function $\rho$ for $\Gamma$.

For arbitrary sets $E$, $F$ and $G$ of $\mathbb{R}^{n}$ we denote by
$\Delta(E, F, G)$ a set of all continuous curves
$\gamma: [a, b] \rightarrow \mathbb{R}^{n}$, that connect $E$ and $F$ in $G$,
i.e.\,, such that $\gamma(a) \in E$, $\gamma(b) \in F$ and $\gamma(t) \in G$
for $a < t < b$.

Let   $D$ be a domain in $\mathbb{R}^{n}$, $n\geqslant2$, $x_0\in D$ and
$d_0 = {\rm dist}(x_{0},
\partial D)$. Set

$$
\mathbb{A}(x_{0}, r_{1}, r_{2}) = \{ x \in \mathbb{R}^{n}: r_{1} < |x-
x_{0}| < r_{2}\} \,,
$$

$$
B(x_0,r) = \{x \in \mathbb{R}^{n}: |x-x_0|< r\} \,,
$$

$$
S_{i}=S(x_0,r_{i})=\{x\in \mathbb{R}^{n}: \,  |x-x_0| = r_{i}\}\,, \quad
i = 1,\,2\,.
$$

The following concept generalizes and localizes the concept of a $Q$-homeomor\-phism
with respect to $p$-modulus.

Let a function  $Q: D\rightarrow [0,\infty]$ be Lebesgue measurable.
We say that a homeomorphism $f: D \rightarrow \mathbb{R}^{n}$ is a {\it ring
$Q$-homeomorphism with respect to $p$-modulus} at $x_0 \in D$, if the
relation

$$
M_p(\Delta(fS_{1}, fS_{2}, fD))\ \leqslant\
\int\limits_{{\mathbb{A}}}Q(x)\,\eta^{p}(|x-x_{0}|)\,dm(x)\,
$$
holds for any ring $\mathbb{A} = \mathbb{A}(x_{0}, r_{1}, r_{2})$\,,
$0< r_{1} < r_{2} < d_0$, $d_0 = {\rm dist}(x_{0},
\partial D)$, and for any measurable function
$\eta: (r_{1}, r_{2}) \rightarrow [0,\infty]$ such that

$$
\int\limits_{r_{1}}^{r_{2}} \eta(r)\, dr = 1\,.
$$
\medskip

The theory of $Q$-homeomorphisms for $p=n$ was developed in
\cite{RS}--\cite{S1}, whereas for $1<p<n$ and $p>n$ see \cite{G1}--\cite{S6} and
\cite{SK1}--\cite{SSK}, respectively.

Denote by $\omega_{n-1}$ the area of the unit sphere
$\mathbb{S}^{n-1} = \{x\in \mathbb{R}^{n}: \, |x| = 1\}$ in $\mathbb{R}^{n}$
and by
$$
q_{x_0}(r)= \frac{1}{\omega_{n-1}\,
r^{n-1}}\int\limits_{S(x_0, r)}Q(x)\, d\mathcal{A}
$$
the integral mean of $Q$ over the sphere $S(x_0,r)=\{x\in \mathbb{R}^{n}: \, |x-x_0| =
r\}$\,, here $d\mathcal{A}$ is the element of the surface area.

\medskip

Now we formulate a criterion which guarantees for a homeomorphism to
be a ring $Q$-homeomorphism with respect to $p$-modulus for $p>1$
in $\mathbb{R}^{n}$, $n\geqslant2$, see Theorem~2.3 in \cite{S2}.

\medskip

\begin{proposition}
Let $D$ be a domain in $\mathbb{R}^{n}$,
$n\geqslant 2$, $x_{0} \in D$, and let $Q:D\to[0,\infty]$ be a Lebesgue measurable
function such that $q_{x_0}(r) \neq \infty$ for a.e.  $r\in(0,
d_0)$,\, $d_0 = {\rm dist}(x_{0}, \partial D)$. A homeomorphism $f:
D \rightarrow \mathbb{R}^{n} $ is a ring $Q$-homeomorphism with respect to
$p$-modulus at $x_{0}$ if and only if the inequality

$$
M_{p}\left(\Delta(fS_{1}, fS_{2}, f\mathbb{A})\right) \leqslant
\frac{\omega_{n-1}}{\left(\int\limits_{r_{1}}^{r_{2}}\frac{dr}{r^\frac{n-1}{p-1}\,q_{x_0}^{\frac{1}{p-1}}(r)}\right)^{p-1}}\,
$$
holds for any $0 < r_{1} < r_{2} < d_0$.
\end{proposition}

We also need the following statement [15, Thm.~1].
\medskip

\begin{theorem}
Let $D$ be a bounded domain in $\mathbb{R}^{n}$,
$n \geqslant 2$ and let $f: D \rightarrow \mathbb{R}^{n}$ be a ring
$Q$-homeomorphism with respect to $p$-modulus at a point $x_0 \in D$
with $p>n$. Assume that the function $Q$ satisfies the growth condition

$$
q_{x_{0}}(t) \leqslant q_{0}\,t^{-\alpha},\, q_{0} \in
(0, \infty)\,,\, \alpha  \in [0, \infty)\,,
$$
for a.e. $t\in (0, d_0)$, $d_0 = {\rm dist}(x_{0},
\partial D)$. Then for all $r\in (0, d_0)$, the estimate

$$
 m(fB(x_0, r)) \geqslant
\Omega_n\,
\left(\frac{p-n}{\alpha+p-n}\right)^{\frac{n(p-1)}{p-n}}q_{0}^{\frac{n}{n-p}}
\, r^{\frac{n(\alpha+ p - n)}{p-n}}\,,
$$
holds,
where $\Omega_n$ is a volume of the unit ball in $\mathbb{R}^{n}$.
\end{theorem}

An upper bound for the area of a disc image
under quasiconformal mappings was established by M.A. Lavrent'ev in his
monograph \cite{Lav}. In \cite[Proposition 3.7]{BGMR}, the
Lavrentiev's inequality was refined in terms of the angular
dilatation. Also in \cite{LS} and \cite{S3}, there were obtained
upper bounds for the area distortion for the ring and lower
$Q$-homeomorphisms. V. Kruglikov has obtained a bound for the
measure of the image of a ball for mappings quasiconformal in the
mean in $\mathbb{R}^{n}$ (see Lemma 9 in \cite{Kru}). In \cite{S2},
there was established an upper bound for the measure of a ball image
under ring $Q$-homeomorphisms with respect to $p$-modulus
$(1<p\leqslant n)$. In \cite{SK1}, \cite{SK2} there were obtained the lower bounds
for the measure of a ball image as $p> n$. Note that the corresponding problems for the
planar case has been studied in \cite{SK3} -- \cite{SK4}.

\medskip

\section{Main results}
\label{sec:2}
In the following theorem, we provide a lower bound for upper limits
of distance ratios under ring $Q$-homeomorphisms with respect to $p$-modulus
with $p>n$.

\medskip

\begin{theorem}
Let $f: \mathbb{B}^{n} \rightarrow \mathbb{R}^{n}$ be a ring
$Q$-homeomorphism with respect to $p$-modulus at a point $x_0 = 0$
with $p>n$ and $f(0) = 0$. Assume that the function $Q$ satisfies the condition

\begin{equation*}
\label{b2H} q_{x_{0}}(t) \leqslant q_{0}\,t^{-\alpha},\, q_{0} \in
(0, \infty)\,,\, \alpha  \in [0, \infty),
\end{equation*}
for a.e. $t\in (0, \varepsilon_0)$, $\varepsilon_0 \in (0, 1)$. Then

\begin{equation}
\label{a1p}
 \limsup\limits_{x\rightarrow 0} \frac{|f(x)|}{|x|^{\frac{\alpha+p-n}{p-n}} } \geqslant \, \left(\frac{p-n}{\alpha+p-n}\right)^{\frac{p-1}{p-n}}
 \, q_{0}^{\frac{1}{n-p}}\,.
\end{equation}
\end{theorem}

\begin{proof} We denote by $L_{f}(r) = \max\limits_{|x| =r} |f(x)|$. Since $f(0) = 0$,\, $\Omega_{n}\,L^{n}_{f}(r) \geqslant m(fB(0, r))$.
Hence,
$$
L_{f}(r) \geqslant \left(\frac{m(fB(0, r))}{\Omega_{n}}\right)^{\frac{1}{n}}\,.
$$

Then, by Theorem~1, we have
$$
L_{f}(r) \geqslant \left(\frac{m(fB(0, r))}{\Omega_{n}}\right)^{\frac{1}{n}} \geqslant
\left(\frac{p-n}{\alpha+p-n}\right)^{\frac{p-1}{p-n}}q_{0}^{\frac{1}{n-p}}
\, r^{\frac{\alpha+ p - n}{p-n}}\,
$$
for any $r \in [0, \varepsilon_0)$.

Thus, we obtain
$$
\limsup\limits_{x\rightarrow 0} \frac{|f(x)|}{|x|^{\frac{\alpha+p-n}{p-n}} } =
\limsup\limits_{r\rightarrow 0} \frac{L_{f}(r)}{r^{\frac{\alpha+p-n}{p-n}} } \geqslant
\left(\frac{p-n}{\alpha+p-n}\right)^{\frac{p-1}{p-n}}q_{0}^{\frac{1}{n-p}}\,.
$$

This completes the proof of Theorem~2.
\end{proof}

Assuming a stronger growth condition on the majorant $Q$ we deduce


\begin{corollary}
Let $f: \mathbb{B}^{n} \rightarrow \mathbb{R}^{n}$ be a ring
$Q$-homeomorphism with respect to $p$-modulus at a point $x_0 = 0$
with $p>n$ and $f(0) = 0$. Assume that the function $Q$ satisfies the condition

\begin{equation*}
\label{b2H} Q(x) \leqslant K\,|x|^{-\alpha},\, K \in
(0, \infty)\,,\, \alpha  \in [0, \infty),
\end{equation*}
for a.e. $x\in B(0, \varepsilon_0)$, $\varepsilon_0 \in (0, 1)$. Then

\begin{equation*}
\label{a1p}
 \limsup\limits_{x\rightarrow 0} \frac{|f(x)|}{|x|^{\frac{\alpha+p-n}{p-n}} } \geqslant \, \left(\frac{p-n}{\alpha+p-n}\right)^{\frac{p-1}{p-n}}
 \, K^{\frac{1}{n-p}}\,.
\end{equation*}
\end{corollary}

\medskip

Letting $\alpha =0$ in Theorem~2, we obtain the following statement.

\medskip

\begin{corollary}
Let $f: \mathbb{B}^{n} \rightarrow \mathbb{R}^{n}$ be a ring
$Q$-homeomorphism with respect to $p$-modulus at a point $x_0 = 0$
with $p>n$ and $f(0) = 0$. Assume that the function $Q$ satisfies the condition

\begin{equation*}
\label{b2H} q_{x_{0}}(t) \leqslant q_{0},\, \, q_{0} \in
(0, \infty),
\end{equation*}
for a.e. $t\in (0, \varepsilon_0)$, $\varepsilon_0 \in (0, 1)$. Then

$$
 \limsup\limits_{x\rightarrow 0} \frac{|f(x)|}{|x|} \geqslant \, q_{0}^{\frac{1}{n-p}}\,.
$$
\end{corollary}

Similarly to Corollary~1, we get


\begin{corollary}
Let $f: \mathbb{B}^{n} \rightarrow \mathbb{R}^{n}$ be a ring
$Q$-homeomorphism with respect to $p$-modulus at a point $x_0 = 0$
with $p>n$ and $f(0) = 0$. If the function $Q(x) \leqslant K$,
for a.e. $x\in \mathbb{B}^{n}$, $K \in (0, \infty)$, then

$$
 \limsup\limits_{x\rightarrow 0} \frac{|f(x)|}{|x|} \geqslant \, K^{\frac{1}{n-p}}\,.
$$
\end{corollary}

Now we verify the sharpness of the bound in our main result.

\medskip

{\it Example.} Let $p > n$ and $q_{0} >0$. Define $f_0:\mathbb{B}^{n} \to \mathbb{R}^{n}$ by

$$
f_0(x)=\begin{cases} q_{0}^{\frac{1}{n-p}} \left(\frac{p-n}{\alpha+p-n}\right)^{\frac{p-1}{p-n}}\,  |x|^{\frac{\alpha+p-n}{p-n}}\, \frac{x}{|x|} \, ,&  x \neq 0\\
0,&  x=0 \,.\end{cases}\,
$$
It can be easily seen that the estimate (\ref{a1p}) is sharp, and it
becomes an equality under $f_0$\,.

Let us show that $f_0$ is a ring $Q$-homeomorphism with
respect to $p$-modulus with $Q(x) =
q_{0}\,|x|^{-\alpha}$ at $x_{0}= 0$ with $p> n$. Clearly,
$q_{x_{0}}(t) = q_{0}\,t^{-\alpha}$. Consider a ring $\mathbb{A}(0,
r_{1}, r_{2})$, $0< r_{1} < r_{2} < 1$. Note that the mapping
$f_{0}$ maps the ring $\mathbb{A}(0, r_{1}, r_{2})$ onto the ring
$\widetilde{\mathbb{A}}(0, \widetilde{r}_{1}, \widetilde{r}_{2})$,
where
$$
\widetilde{r}_{i} = q_{0}^{\frac{1}{n-p}}\,
\left(\frac{p-n}{\alpha+p-n}\right)^{\frac{p-1}{p-n}}\,
r_{i}^{\frac{\alpha+p-n}{p-n}}\,, \quad i = 1,\,2.
$$
Denote by $\Gamma$ a set of all curves that join the spheres
$S_{1} = S(0,r_{1})$ and $S_{2} = S(0,r_{2})$ in the ring $\mathbb{A} = \mathbb{A}(0, r_{1},
r_{2})$. Then one can calculate the $p$-modulus of the family of curves
$f_{0}\Gamma$ in an implicit form:

$$
M_{p}(f_{0}\Gamma) = \omega_{n-1}\,
\left(\frac{p-n}{p-1}\right)^{p-1}\,
\left(\widetilde{r}_{2}^{\frac{p-n}{p-1}} -
\widetilde{r}_{1}^{\frac{p-n}{p-1}}\right)^{1-p}\,
$$
(see, e.g., (2) in \cite{Ge}).

Substituting the values $\widetilde{r}_{1}$
and $\widetilde{r}_{2}$, defined above, in the previous equality one gets
 $$
 M_{p}(f_{0}\Gamma) =
 \omega_{n-1}\, q_{0}\, \left(\frac{\alpha+p-n}{p-1}\right)^{p-1}\, \left(r_{2}^{\frac{\alpha+p-n}{p-1}} -
 r_{1}^{\frac{\alpha+p-n}{p-1}}\right)^{1-p}\,.
$$

Note that the last equality can be written as
$$
M_{p}(f_{0}\Gamma) =
\frac{\omega_{n-1}}{\left(\int\limits_{r_{1}}^{r_{2}}\frac{dt}{t^\frac{n-1}{p-1}\,q_{x_0}^{\frac{1}{p-1}}(t)}\right)^{p-1}}\,,
$$
where $q_{x_0}(t) = q_{0}\,t^{-\alpha}$ and $x_{0}= 0$.

Hence, by Proposition 1, the homeomorphism $f_{0}$ is a ring
$Q$-homeomorphism with respect to $p$-modulus for $p>n$ with
$Q(x) = q_{0}\,|x|^{-\alpha}$ at origin.

{\it Acknowledgments}. This work was supported by a grant from the Simons Foundation (1030291,
R.R.S., B.A.K.).

%
%
%

\end{document}